\documentclass[11pt,twoside]{article}
\usepackage[utf8]{inputenc}
\usepackage[paper=a4paper, margin=1in]{geometry}

\usepackage{amsmath}
\usepackage{amssymb}
\usepackage{amsfonts}
\usepackage{amsthm}
\usepackage[backend=biber,style=alphabetic,citetracker=true, sorting=nyt]{biblatex}
\DeclareFieldFormat{postnote}{#1}
\DeclareFieldFormat{multipostnote}{#1}

\usepackage[utf8]{inputenc}
\usepackage{tikz-cd}
\usepackage{tikz}
\usetikzlibrary{patterns,shapes,arrows}
\usepackage{array}
\usepackage{setspace}
\usepackage{url}
\usepackage{enumerate}
\usepackage{comment}
\usepackage{graphicx}
\usepackage{hyperref}
\hypersetup{colorlinks=true,linkcolor=black,citecolor=black,urlcolor=black}
\usepackage{caption}
\usepackage{subcaption}

\makeatletter
\newtheorem*{rep@theorem}{\rep@title}
\newcommand{\newreptheorem}[2]{%
	\newenvironment{rep#1}[1]{%
		\def\rep@title{#2 \ref{##1}}%
		\begin{rep@theorem}}%
		{\end{rep@theorem}}}
\makeatother

\newtheorem{lemma}{Lemma}
\newtheorem{theorem}[lemma]{Theorem}
\newreptheorem{theorem}{Theorem}
\newtheorem{prop}[lemma]{Proposition}
\newreptheorem{prop}{Proposition}

\newreptheorem{corollary}{Corollary}

\theoremstyle{definition}

\newtheorem{example}[lemma]{Example}
\newreptheorem{example}{Example}

\theoremstyle{remark}

\newcommand{\ints}{\mathbb{Z}}

\newcommand{\rationals}{\mathbb{Q}}

\newcommand{\Frob}{\operatorname{Frob}}

\newcommand{\End}{\operatorname{End}}

\newcommand{\Span}{\operatorname{Span}}

\newcommand{\id}{\operatorname{id}}

\title{Every $7$-Dimensional Abelian Variety over $\rationals_p$ has
a Reducible $\ell$-adic Galois Representation }
\author{Lambert A'Campo}
\date{\today}

\begin{document}

\maketitle

\begin{abstract}
Let $K$ be a complete, discretely valued field with finite residue field
and $G_K$ its absolute Galois group.
The subject of this note is the study of the
set of positive integers $d$ for which there exists an absolutely irreducible $\ell$-adic representation of
$G_K$ of dimension $d$ with rational traces on inertia. 
Our main result is that non-Sophie Germain primes are not in this set
when the residue characteristic of $K$ is $> 3$.
The result stated in the title is a special case.
\end{abstract}

Over a number field one expects a `generic' abelian variety to
be irreducible. For instance \cite{zarhin1999} proves
a result in this direction and provides us with abelian varieties
$A$ of any dimension such that $\End(A) = \ints$. 
By Faltings' Isogeny Theorem the Tate module of such an abelian variety is
an absolutely irreducible Galois representation.

Over local fields with residue characteristic $p$,
the situation is very different. The dimensions which 
appear in absolutely irreducible $\ell$-adic Galois representations
with rational traces on inertia form a proper
subset of the positive integers. 
We are not able to give a full description
of this subset, however we prove some restrictions, namely we show that if
the representation is tamely ramified, then its dimension is
a value of the Euler totient function, see Proposition \ref{phim_prop}. 
In the case of wild ramification
we can only conclude that $(p-1)$ divides the dimension, see Proposition \ref{wildp_prop}.
Nevertheless,
this is still enough to show

\begin{theorem} \label{abelian_var_dim_thm}
		Let $p \neq 2,3$, $K/\rationals_p$ a finite extension
		and $\ell \neq p$ a prime. If $d$ is a prime such that $2d + 1$ is not prime, then there is no abelian variety $A/K$ of dimension $d$
	whose associated Galois representation \[
	V_\ell A := \left( \lim_{\substack{\longleftarrow \\ n \geq 1}} A[\ell^n](\overline{K}) \right) \otimes_{\ints_\ell} \rationals_\ell
	\]
	is absolutely irreducible. In particular, the conclusion holds for
	all primes $d \equiv 1 \pmod{3}$.
\end{theorem}

\begin{example}
	Let $C/\rationals_5$ be the genus 7 hyperelliptic curve $y^2 = x^{15} - 5x$. Then the Jacobian of $C$ is a 7-dimensional abelian variety over $\rationals_5$ and has a reducible $\ell$-adic Galois representation because $15 = 2 \cdot 7 + 1$ is not prime. 
	
	In this special example we can verify the theorem using the extra automorphism
	$\alpha : (x,y) \mapsto (-x, i y)$, where $i \in \rationals_5$ is a
	root of $x^2 + 1$. One checks that
	$\alpha^2 + 1 = 0$ without difficulty. By \cite[n° 69, Corollaire 2]{weil48},
	the characteristic polynomial of $\alpha$ has rational coefficients and thus is a power of 
	$x^2 + 1$ and by counting dimensions it must be $(x^2 + 1)^7$.
	Hence, we have the non-trivial decomposition $V_\ell A \otimes_{\rationals_\ell} \rationals_\ell(i) = \ker (\alpha - i) \oplus \ker(\alpha + i)$ into $7$-dimensional summands.
\end{example}

A \emph{Sophie Germain prime} is a prime number $d$ such that $2d + 1$ is also prime. Consequently, Theorem \ref{abelian_var_dim_thm} is concerned with \emph{non-Sophie German primes}. The first few are $d = 7,13,17,19,31 \dots$
\cite[{\href{https://oeis.org/A053176}{A053176}}]{OEIS}.
The set of non-Sophie Germain primes
is infinite since for any prime $d \equiv 1 \pmod{3}$, 3 divides $2d + 1$.
Thus every prime $\equiv 1 \pmod{3}$ is not a Sophie Germain prime and the 
theorem applies.
Moreover, non-Sophie Germain primes should contain 100\% of all primes
by the following probabilistic heuristic. The events `$n$ is prime' and `$2n + 1$ is prime' have probabilities $1/\log n$ and
$1/\log(2n + 1)$, respectively. They are not quite independent but one 
can analyse the dependence and predict the probability that $n$ is a
Sophie Germain prime to be $C/(\log n \log(2 n + 1))$
 for some constant $C$ \cite{bateman_horn}.
Thus the number of Sophie Germain primes
up to $x$ is of order
$
Cx/(\log x)^2
$
which is easily seen to be 0\% of primes as $x \to \infty$ by the prime number theorem.

Our results continue the observation \cite[footnote 2]{dokchitser_surjective}
which shows that it is more difficult to attack the 
inverse Galois problem for the 
groups $\operatorname{GSp}_{2n}(\mathbb{F}_p)$ by constructing
an abelian variety over $\rationals$ with suitable reductions 
as one can easily do with elliptic curves for $n = 1$.

We fix some notations. Let
$K$ be a complete, discretely valued field with finite residue field
$k$ of characteristic $p$ and cardinality $q$. So $K$
is either a finite extension of $\rationals_p$ or a finite extension of
$\mathbb{F}_p((t))$. The former is the most interesting case since
in the latter all ramification is tame.
The absolute Galois group $G_K$ of the field $K$ is a split extension
\[
1 \to I_K \to G_K \to G_k \to 0,
\]
where $I_K$ is the inertia group of $K$ \cite[I. \S 7]{cassels_ant}.
The representations we consider are
continuous representations $\rho: G_K \to \operatorname{GL}_n(\overline{\rationals_\ell})$, where $\ell \neq p$
and for all $g \in I_K$,
$\operatorname{tr}(\rho(g)) \in \rationals$. In particular,
this should include the class of `geometric' $\ell$-adic 
representations which occur as the Tate module 
of an abelian variety $A/K$ or more generally as the $\ell$-adic
cohomology of a variety $X/K$. For example see \cite[Theorem 2]{serre_tate_good_reduction} for a proof that 
the Tate module of
an abelian variety over $K$ with potentially good reduction
has rational traces on inertia. 
With these notations, our precise results are the two following propositions.

\begin{prop} \label{phim_prop}
	Suppose that $V$ is a continuous, irreducible,
	tamely ramified 
	$\overline{\rationals_\ell}$-representation of $G_K$ such that the 	trace of any $h \in I_K$ 
	is rational.
	Then either $\dim V = 1,2$ or there exists an odd prime $v \neq p$, 
	such that $\dim V = (v-1)v^a$, for some $a \geq 0$ such that
	$q$ generates $(\ints/v^{a+1}\ints)^\times$. In particular,
	$\dim V = \varphi(m)$ for some positive integer $m$.\footnote{See 
		\cite[{\href{http://oeis.org/A005277}{A005277}}]{OEIS} for
		examples of even integers which are not a value of $\varphi$.}
	Moreover, each of these dimensions is realised by such a representation.
\end{prop}

\begin{prop} \label{wildp_prop}
	Suppose that $V$ is a continuous, irreducible,
	wildly ramified
	$\overline{\rationals_\ell}$-representation of $G_K$ such that
	the trace of any $h \in I_K$ is rational.
	Then $(p -1) \mid \dim V$.
\end{prop}

For the proofs we fix a geometric Frobenius element $\phi_K \in G_K$,
i.e. an element which reduces to $\Frob_k^{-1} \in G_k$,
where $\Frob_k(x) = x^q$ and $q = |k|$. Moreover, recall
that $I_K$ has a unique pro-$p$ Sylow subgroup $P_K$,
called the wild inertia group which is also a normal subgroup of $G_K$.
There is an isomorphism $I_K/P_K \cong \prod_{v \neq p} \ints_v$
such that $\phi_K^{-1} x \phi_K = x^q$ for all $x \in I_K / P_K$
and we fix a projection $t_\ell : I_K/P_K \to \ints_\ell$ which is
the maximal pro-$\ell$ quotient of $I_K/P_K$. See \cite[I.\S 8]{cassels_ant} for proofs of these facts. 

\begin{proof}[Proof of Proposition \ref{phim_prop}]
	Since $G_K/P_K$ is topologically finitely generated, 
	we can assume that $V$ is defined over
	a finite extension $F/\rationals_\ell$.
	Then the Grothendieck Monodromy Theorem \cite[Appendix]{serre_tate_good_reduction} shows
	that there is an open subgroup $H < I_K$ 
	and a nilpotent operator $N \in \End(V)$ such that
	$h \in H$ acts as $\exp(t_\ell(h) N)$
	and $\phi_K^{-1} N \phi_K = q N$. 
	
	We now show that in fact $N = 0$. By possibly shrinking $H$ we may assume that it is normal in $G_K$. Now $V^H = \ker N$ is a subrepresentation of $V$ and either $\ker N = 0$ or $\ker N = V$ since $V$ is irreducible. The former contradicts the fact that $N$ is nilpotent and hence $\ker N = V$ and $N = 0$. Consequently, $I_K$ acts through the finite quotient $I_K/H$.
	
	The tame inertia group
	is pro-cyclic and so $I_K$ acts on $V$ 
	through $\ints/m\ints$ for some
	integer $m$ which we choose to be minimal. 
	Let $\tau \in I_K$ be a generator of this action.
	The eigenvalues of $\rho(\tau)$ are $m$th roots of unity.
	Suppose one of the eigenvalues is not a primitive root of unity.
	Then there is $n < m$ such that $W := \{ v \in V : \rho(\tau)^n v = v \} \subset V$ is non-zero. Clearly, $W$ is stable under multiplication by $\rho(\tau)$. The relation $\tau^n \phi_K = \phi_K \tau^{qn}$ 
	implies that $W$ is also stable under $\phi_K$: If
	$w \in W$, then 
	\[
	\rho( \tau^n) \rho(\phi_K) w = \rho(\phi_K \tau^{qn}) w =  \rho(\phi_K) w,
	\]
	 i.e. $\rho(\phi_K) w \in W$. Thus $W$ is a subrepresentation of $V$. But $W \neq V$
	since $\rho(\tau)$ has order $m$ by the minimality of $m$.
	This contradicts the irreducibility of $V$.
	
	Thus all eigenvalues of $\rho(\tau)$ are primitive $m$th roots
	of unity and so $\det(X \id - \rho(\tau)) = \Phi_m(X)^t$
	for some $t$, where $\Phi_m$ is the $m$th cyclotomic polynomial.
	Moreover, finite group representation theory applied to $\ints/m\ints$ shows that $\rho(\tau)$ is diagonalisable.
	So $V = \bigoplus_{\lambda} V_\lambda$, where
	the direct sum is over primitive $m$th roots of unity and
	$V_\lambda$ is the $\lambda$-eigenspace of $\rho(\tau)$ and
	$\dim V_\lambda = t$.

	Let $\lambda$ be a primitive $m$th root of unity.
	The relation $\phi_K^{-1} \tau \phi_K = \tau^q$
	shows that $\phi_K$ maps $V_\lambda$ to $V_{\lambda^q}$.
	Let $r$ be the order of $q$ in $(\ints/m\ints)^\times$ so 
	that $r$ is the smallest positive integer with $\phi_K^r V_\lambda = V_\lambda$.
	Choose $v \in V_\lambda \setminus \{0\}$ to be an eigenvector of 
	$\phi_K^{r}$. Then the subspace
	$U = \Span( \{\phi_K^n v : n \in \ints \}) \subset V$
	is a subrepresentation of dimension $r$.
	By the irreducibility of $V$, we find $U = V$.
	Comparing dimensions yields
	\[
	t |(\ints/m\ints)^\times| = \dim V = \dim U = r.
	\]
	So $t = 1$ and $q$ is a generator of $(\ints/m\ints)^\times$.
	The group $(\ints/m\ints)^\times$ is cyclic only when $m = 2,4$, $m = v^{a+1}$ or $m = 2 v^{a+1}$,
	where $v$ is an odd prime. Hence $\dim V = 1,2$ or $\dim V = (v-1)v^a$.
	
	Moreover, for every such integer $m$ which is coprime to $p$, there
	is a unique quotient of the tame inertia group of order $m$
	and we can define a representation $V$ of $G_K$ by the space 
	of functions
	$V = \{f : (\ints/m\ints)^\times \to \overline{\rationals_\ell}\}$,
	where $(\phi_K f)(x) = f(q x)$ and
	$(\tau f)(x) = \lambda(x) f(x)$ and
	$\lambda : \ints/m\ints \to \overline{\rationals_\ell}^\times$ is any monomorphism.
\end{proof}

\begin{lemma} \label{lattice_lemma}
	Let $G$ be a compact group and $\rho : G \to \operatorname{GL}_n(F)$
	a continuous homomorphism, where $F$ is a discretely valued field
	of characteristic $0$ with valuation $v : F^\times \to \ints$
	and ring of integers $R = \{ x \in F : v(x) \geq 0\}$.
	Then $\rho$ is conjugate to 
	a continuous homomorphism $G \to \operatorname{GL}_n(R)$.
\end{lemma}

\begin{proof}
	Note that there are no integers between $0$ and $-1$ and hence
	$R = \{ x \in F : v(x) > -1/2\}$ is an open subset of $F$.
	Consequently $\operatorname{GL}_n(R)$ is an open subgroup of
	$\operatorname{GL}_n(F)$ and its preimage $H = \rho^{-1}(\operatorname{GL}_n(R))$ is open as well.
	By the compactness of $G$, $G/H$ is finite and so 
	$G R^n \subset F^n$ is a finitely generated, torsion free $R$-submodule which is
	$G$-invariant by definition. Since $R$ is a discrete valuation ring,
	$G R^n$ is free of rank $n$. So with respect to a basis of $G R^n$,
	$\rho$ takes values in $\operatorname{GL}_n(R)$.
\end{proof}

\begin{proof}[Proof of Proposition \ref{wildp_prop}] 
	The result is trivial for $p = 2$ so we assume $p \geq 3$.
	Then by \cite{jannsen_wingberg_1982}, $G_K$ is topologically finitely 
	generated and so $V$ is defined over
	a finite extension $F/\rationals_\ell$.\footnote{In practice one does not need the strong result of \cite{jannsen_wingberg_1982} since most representations already come defined over a finite extension of $\rationals_\ell$. For example all Tate modules are defined over $\rationals_\ell$.}
	By Lemma \ref{lattice_lemma}, we can assume that
	$G_K$ acts on $V$ by a continuous homomorphism
	$\rho : G_K \to \operatorname{GL}_n(\mathcal{O}_F)$, where $n = \dim V$
	and $\mathcal{O}_F$ is the ring of integers of $F$.
	The reduction $\overline{\rho} : G_K \to \operatorname{GL}_n(\mathcal{O}_F/\mathfrak{m}_F \mathcal{O}_F)$ has finite image. As the kernel of
	$\operatorname{GL}_n(\mathcal{O}_F) \to
	\operatorname{GL}_n(\mathcal{O}_F/\mathfrak{m}_F \mathcal{O}_F)$
	is a pro-$\ell$-group we conclude that
	$P_K$ acts faithfully through a finite quotient $G$ on $V$.
	Write $V = \bigoplus_{i} V_i$ as a direct sum of irreducible
	representations of $G$.

	
	Note that $V^G \subset V$ is a subrepresentation since $P_K < G_K$ 
	is a normal subgroup.
	Thus $V^G = 0$ since otherwise $V$
	is tamely ramified. Consequently 
	all the $V_i$ are non-trivial.
	Let $G_i = G/ \ker \rho_i$ be the quotient of $G$ which acts
	faithfully on $V_i$.
	Since $G_i$ is a $p$-group, there is 
	a non-trivial element $g_i$ in the center of $G_i$ 
	which acts as a scalar $\lambda_i$ on $V_i$ by Schur's Lemma.
	As $G_i$ acts faithfully we find that $\lambda_i$
	is a non-trivial $p^t$th root of unity for some $t \geq 1$.
	Hence the cardinality of the Galois orbit of the character of $V_i$ 
	is a multiple of $p^{t-1} (p-1)$. Since the character
	of $V$ is defined over $\rationals$, this implies that
	the decomposition of $V$ contains all these Galois conjugates
	and in particular that $(p-1) \mid \dim V$.
\end{proof}

\begin{proof}[Proof of Theorem \ref{abelian_var_dim_thm}]
	Suppose $A/K$ is an abelian variety of dimension $d$,
	where $d$ is a prime such that $2d + 1$ is not prime
	and $V_\ell A$ is absolutely irreducible.
	To apply the results above we need to check that each 
	$h \in I_K$ has rational trace on $V_\ell A$.
	Combining the irreducibility assumption with the Grothendieck Monodromy
	Theorem \cite[Appendix]{serre_tate_good_reduction}, we see that
	$I_K$ must act
	through a finite quotient on $V_\ell A$ as in
	the proof of Proposition \ref{phim_prop}.
	By the N\'eron-Ogg-Shafarevich criterion, $A$ has potentially
	good reduction and the characteristic polynomials of 
	elements $h \in I_K$ have integral coefficients independent of $\ell$
	by \cite[Theorem 2]{serre_tate_good_reduction}.
	
	Either $V_\ell A$ is at most tamely ramified or it is wildly 
	ramified. Thus either Proposition \ref{phim_prop} or
	Proposition \ref{wildp_prop} applies to $V_\ell A$.
	In the former case, we have
	$2d = \dim V_\ell A = \varphi(m)$ for some $m$ which is impossible since $2d + 1$ is not prime.
	In the latter case we have $(p - 1) \mid 2d$, so $p \in \{2,3,d+1,2d+1\}$ contradicting the hypothesis that $p \neq 2,3$ and $2d + 1$ is not prime.
\end{proof}

\noindent \textbf{Acknowledgments.} I thank Vladimir Dokchitser for suggesting this subject to me and for encouraging me to write this note. He patiently discussed many (wrong) versions of Proposition \ref{wildp_prop} with me and provided lots of essential advice and ideas.
Moreover, I thank my father and Jesse Pajwani for reading an early version of this text.

\printbibliography

\end{document}